\documentclass[11pt]{amsart}
\usepackage{latexsym,amsfonts,amssymb,amsthm,amsmath,mathrsfs,color,amscd,graphicx,fullpage,fancyhdr,parskip,hyperref}

\newcommand{\s}[1]{{\mathcal #1}}

\newcommand{\bb}[1]{{\mathbb #1}}

\newtheorem{theorem}{Theorem} %[section]
\newtheorem{corollary}[theorem]{Corollary}

\newtheorem{lemma}[theorem]{Lemma}
\newtheorem{proposition}[theorem]{Proposition}

\newtheorem{definition}[theorem]{Definition}

\newtheorem{remark}[theorem]{Remark}
\newtheorem{assumption}[theorem]{Assumption}

\numberwithin{equation}{section}
\numberwithin{theorem}{section}

\begin{document}

\title{Weak solutions for mean field games with congestion}
\author{P. Jameson Graber}
\thanks{Research supported by NSF grant DMS-1303775.
The author wishes to thank Alain Bensoussan for his mentoring and support.}
\address{Naveen Jindal School of Management\\
The University of Texas at Dallas\\
800 West Campbell Rd, SM30\\
Richardson, TX 75080-3021\\
Phone: (972) 883-6249}
\email{pjg140130@utdallas.edu}
\dedicatory{Version: \today}

\maketitle
%\tableofcontents

\begin{abstract}
We study the short-time existence and uniqueness of solutions to a coupled system of partial differential equations arising in mean field game theory.
It has the generic form
$$
\left\{
\begin{array}{c}
-\partial_t u - \Delta u + H(t,x,m,\nabla u) = f(t,x,m)  \\
\partial_t m - \Delta m - \mathrm{div} \left(m\nabla_p H(t,x,m,\nabla u)\right) = 0
\end{array}\right.
$$
plus initial-final and boundary conditions.
The novelty of the problem is that the Hamiltonian $H(t,x,m,p)$ may take such forms as $m^{-\alpha}|p|^r$ for some $\alpha \geq 0$ and $r > 1$.
Our main result is the existence of weak solutions for small times $T$ so long as $r$ is not too large, and uniqueness under additional constraints.
The main ingredient in the proof is an a priori estimate on solutions to the Fokker-Planck equation.
We also briefly consider existence and uniqueness of solutions to an optimal control problem related to mean field games.


\end{abstract}

\section{Introduction}

Let $\Omega \subset \bb{R}^d$ be a bounded domain, let $T > 0$ be given, and set $Q  := \Omega \times (0,T).$
Our purpose is to study
\begin{equation}\label{mfg}\left\{
\begin{array}{cll}
\vspace{5pt}
(i) & -\partial_t u - \Delta u + H(t,x,m,\nabla u) = f(t,x,m) &\text{in}~~Q  \\
\vspace{5pt}
(ii) & \partial_t m - \Delta m - \mathrm{div} \left(m\nabla_p H(t,x,m,\nabla u)\right) = 0 &\text{in}~~Q  \\
(iii) & u(T,x) = g(x,m(T,x)),~~ m(0,x) = m_0(x) &\text{in}~~\Omega 
\end{array}\right.
\end{equation}
equipped with Dirichlet, Neumann, or periodic boundary conditions.
Here we assume that $H(t,x,m,p)$ has a structure of which the canonical example is
\begin{equation}
\label{canonical}
H(t,x,m,p) = \frac{|p|^r}{rm^\alpha}, ~~ r > 1, ~ \alpha \geq 0.
\end{equation}

The main result of this paper is the existence of short-time solutions for system \eqref{mfg} under the condition that $r$ is not ``too big" (see Section \ref{assumptions}).

System \eqref{mfg} represents a mean field game with a congestion term.
Heuristically, the function $u$ denotes the value function of a representative player whose objective is to minimize
\begin{equation}
\bb{E}\left[\int_t^T L(s,X(s),m(s,X(s)),v(s,X(s))) + f(s,X(s),m(s,X(s)))ds + g(X(T),m(T,X(T))) \right]
\end{equation}
over all trajectories $X = X(t)$ of the controlled stochastic differential equation
\begin{equation}
dX(s) = v(s,X(s))ds + \sqrt{2}dB(s), ~~ X(t) = x.
\end{equation}
The Lagrangian $L$ is defined by taking the Legendre transform of $H$ in the last coordinate, the canonical example being given by 
$$
L(t,x,m,v) = \frac{m^\alpha |v|^{r'}}{r'}, ~~ \frac{1}{r} + \frac{1}{r'} = 1.
$$
Here $m(t,x)$ is the density of participants in the game, while $\alpha$ can be taken as a congestion parameter.
Increasing $\alpha$ will increase the relative cost incurred by the representative player when moving at high speeds in areas of high population density.

The main result of this paper is the existence and uniqueness of weak solutions of \eqref{mfg} under general assumptions on the data.
Our main contribution is to show that $H(t,x,m,\nabla u)$ may have a highly singular dependence on $m$, provided that it does not blow up too quickly as $\nabla u$ becomes large.
This condition is made precise in Section \ref{assumptions} below.

To put this into perspective, we note that the vast majority of existence results for mean field games address only the case where $H$ does not depend on $m$ at all, that is, $H = H(t,x,\nabla u)$.
Using a priori estimates and compactness criteria, one can construct weak solutions using known results for quasilinear parabolic equations \cite{lasry07,porretta2013weak}.
The other option in this case is to view \eqref{mfg} as a condition of optimality for the control of the Fokker-Planck equation (or, by duality, of the Hamilton-Jacobi equation).
This was pointed out in \cite{lasry07} and exploited in \cite{cardaliaguet2013weak,graber2014optimal,cardaliaguet2014mean,cardaliaguet2014second} to obtain new existence results even for systems with degenerate diffusion (that is, replacing the Laplacian $-\Delta$ with an elliptic differential operator which is not strictly positive).
In such works, the loss of regularity is made up for by the fact that solutions can be constructed by passing to the limit on optimizing sequences, which, with the help of a priori estimates on Hamilton-Jacobi equations, can be shown to be compact in appropriately chosen functions spaces.
See \cite{cardaliaguet2014second} for details.

Up until now, very few authors have studied PDE systems of mean field games which did not fit this paradigm.
The motivation for the present work is a result of Gomes and Mitake in \cite{gomes2014existence}, where they prove the existence of classical solutions for a stationary version of \eqref{mfg} with the Hamiltonian given by the canonical example \eqref{canonical} with $r = 2$.
Their proof relied on a priori estimates which emerge from the specific structure they consider.
In particular, integration by parts yields a special cancellation which allows one to conclude that $1/m$ is bounded, after which everything else follows from classical methods for parabolic equations.

A system which resembles the mean field game \eqref{mfg} where $H$ has the form \eqref{canonical} appears in \cite{cardaliaguet2012geodesics}.
However, in that reference the problem is one of optimal control, and the PDE system studied there is a condition of optimality and not a game.
Such optimal control problems arise in transport theory \cite{dolbeault2009new} and also have applications to  traffic flow and congestion \cite{burger2013mean}.
It is interesting that the PDE systems for such control problems resemble \eqref{mfg}.
However, the fact that our case is not a control problem creates added difficulties which we must surmount.
See \cite{bensoussan13a} for related remarks on the comparison between mean field games and mean field optimal control problems.
We will briefly visit the mean field type control problem corresponding to \eqref{mfg} in Section \ref{section mean field control}.

Rather than considering only the canonical Hamiltonian given by \eqref{canonical}, our main results cover much more general Hamiltonians which may depend on $m$ in a highly singular way.
The difficulty we face when constructing a solution is showing compactness in the Hamiltonian term.
Classical methods for quasilinear parabolic equations (see \cite{porretta1999existence} and references therein) do not apply here, since $H(t,x,m,\nabla u)$ could blow up due to $m$ rather than $\nabla u$.
In this work we deal with this by finding conditions under which a solution to \eqref{mfg} will satisfy the a priori estimate
$$
\|m\|_\infty + \|1/m\|_\infty \leq C,
$$
which then allows the classical methods to go through.
Such an estimate can be proved for solutions of the Fokker-Planck equation provided that the vector-field generating the flow is sufficiently integrable.
This is the main reason for assuming that $H$ does not grow ``too fast" with respect to the variable $\nabla u~$: with this requirement, the energy estimates for mean field games yield that $\nabla_p H$ is sufficiently integrable.

Another work by Gomes and Voskanyan  \cite{gomes2015short} studies a  model similar to \eqref{mfg}.
There they consider the existence of smooth solutions, for Hamiltonians of the form \eqref{canonical} having any subquadratic growth in the gradient variable but for a sufficiently small exponent on the density: that is, any $1 < r < 2$ by for $\alpha$ sufficiently small.
The result holds, as in the present work, for sufficiently small times $T$.
These two works can be seen as complementary insofar as (i) the present work does not study bounded, smooth solutions but rather weak solutions, and (ii) the present work studies existence for $\alpha$ in any range (but with restrictions on $r$).

The author wishes to thank Diogo Gomes for helpful comments made during the preparation of this article.
 
To conclude this introduction, we outline the structure of the paper.
The section immediately following this introduction gives some notation, defines weak solutions and presents our main results, Theorems \ref{existence of weak solutions} and \ref{uniqueness}.
Section \ref{existence section} is devoted to proving the existence of solutions.
The largest part of this section is devoted to obtaining a priori estimates on the Fokker-Planck equation--see Proposition \ref{fokker planck bounds}.
Section \ref{uniqueness section} is devoted to proving uniqueness of solutions under certain conditions.
Finally, in Section \ref{section mean field control} we briefly examine the problem of mean field type control.
It is shown in particular that uniqueness of solutions for the control problem holds under much more general conditions than for mean field games.
The assumptions in play in the present work make it easier to construct solutions of the mean field type control problem; we can expect that in future work, much more general results will be proved using different variational techniques.

\section{Preliminaries}

\subsection{Notation and assumptions}
\label{assumptions}

We denote by $L^p(Q ) = L^p((0,T) \times \Omega)$ or just $L^p$ the Lebesgue space of $p$-integrable functions on $Q $, with norm $\|\cdot\|_{L^p}$ or simply $\|\cdot\|_p$.
We denote by $\nabla$ the $d$-dimensional spatial gradient.
Sobolev spaces will be labeled $W^{s,p}$, or $H^s$ when $p = 2$, with corresponding norms $\|\cdot\|_{W^{s,p}}$.
We will write $\|\cdot\|_{L^t_pL^q_x}$ for the norm in $L^p(0,T;L^q(\Omega))$, or more generally $\|\cdot\|_{W^{s,p}_tW^{r,q}_x}$ for the norm in $W^{s,p}(0,T;W^{r,q}(\Omega))$.
The space of continuous functions on $Q$ will be denoted $C(Q)$, while $C^s(Q)$ denotes the space of all functions with $s$ continuous derivatives.
Analogous notation applies for function spaces on $\Omega$.
If $X$ is any function space, then the space of all non-negative functions in $X$ will be denoted $X_+$.
Throughout this work the symbol $C$ will denote a constant depending only on the parameters given by the theorem or proposition being proved; its precise value may change from line to line.

We assume the following hypotheses on the data:
\begin{enumerate}
\item (Initial condition) $m_0 \in L^\infty(\Omega)$, $m_0 \geq c > 0$, and $\int_Q m_0 = 1$
\item (Coupling) $f(t,x,m)$ and $g(x,m)$ are measurable in $(t,x)$, continuous and nondecreasing in $m$, and bounded below.
Moreover,
\begin{equation} \label{fL and gL defined}
f_L(t,x) := \sup_{m \in [0,L]} f(t,x,m), ~~ g_L(x) := \sup_{m \in [0,L]} g(x,m)
\end{equation}
are both integrable for all $L > 0$.
\item (Hamiltonian) We assume that $H = H(t,x,m,p)$ is measurable in $(t,x)$, continuous in $m$, and both convex and continuously differentiable in $p$.
Moreover, we assume that for some constants $C> 0,\lambda \geq 0$,
\begin{align}
& H(t,x,m,p) \geq \frac{|p|^r}{C(m^\lambda+m^{-\lambda})} \label{H positive}\\
& |\nabla_p H(t,x,m,p)| \leq C(m^\lambda+m^{-\lambda})(1 + |p|^{r-1}) \label{DH bounded}\\
& \nabla_p H(t,x,m,p) \cdot p - rH(t,x,m,p) \geq -C \label{DH bounded below}
\end{align}
The exponent $\lambda$ appearing in four different places in \eqref{H positive}-\eqref{DH bounded} could in principle be four different exponents, but it is sufficient for our purposes to take $\lambda$ to be their maximum.
The main point is that $H$ and $\nabla_p H$ can blow up polynomially in $m$ or $m^{-1}$, that is for $m$ large or small.
\item (Growth of Hamiltonian) The exponent $r$ setting the growth rate of $H$ satisfies the condition
\begin{equation}
\label{growth rate of Hamiltonian}
r < \frac{d+2}{d+1}.
\end{equation}
\end{enumerate}
The canonical example presented in the introduction is $H(t,x,m,p) = m^{-\alpha}|p|^r$, which satisfies the above hypotheses with $\lambda = \alpha$.
However, we can allow much more general Hamiltonians, for instance those having the form
$$
H(t,x,m,p) = h(t,x,m)|p|^r
$$
where $h$ is continuous in $m$ and satisfies
$$
\epsilon(m^\lambda + m^{-\lambda})^{-1} \leq h(t,x,m) \leq C(m^\lambda + m^{-\lambda})
$$
for some positive constants $C$, $\epsilon$,  and $\lambda$.
Note that $\lambda$ may be large,
so that $h$ may be very large or quickly vanish when $m$ is either large or small.

\subsection{Definition of solutions}

We now define solutions of system \eqref{mfg}.
\begin{definition}
\label{definition of solutions}
A couple $(u,m) \in L^1(Q) \times L^1(Q)_+$ is a weak solution of \eqref{mfg} provided
\begin{enumerate}
\item[(i)] $m \in C([0,T];L^1(\Omega))$ and $g(\cdot,m(T)) \in L^1(\Omega)$;
\item[(ii)] $H(\cdot,\cdot,m,\nabla u), ~ (1+m)|\nabla_p H(\cdot,\cdot,m,\nabla u)|^2, ~ f(\cdot,\cdot,m) \in L^1(Q)$;
\item[(iii)] the equations hold in the sense of distributions,
\begin{equation}
\int_Q u (\phi_t - \Delta \phi) + \int_Q  H(t,x,m,\nabla u)\phi 
= \int_Q f(t,x,m)\phi + \int_{\Omega} g(x,m(T))\phi(T)
\end{equation}
for every $\phi \in C_c^\infty((0,T] \times \overline{\Omega})$ satisfying \eqref{test function boundary}, and
\begin{equation}
\int_Q m(-\phi_t -\Delta \phi) + \int_Q m\nabla_p H(t,x,m,\nabla u) \cdot \nabla \phi = \int_{\Omega} m_0 \phi(0)
\end{equation}
for every $\phi \in C_c^\infty([0,T) \times \overline{\Omega})$ satisfying \eqref{test function boundary}, where the boundary conditions may be one of the following:
\begin{equation}
\label{test function boundary}
\begin{array}{ll}
\text{periodic:}& \phi ~\text{is}~\bb{Z}^d-\text{periodic in}~x,\\
\text{Dirichlet:}& \phi = 0 ~\text{on}~ \Sigma := (0,T) \times \partial \Omega, ~\text{or}\\
\text{Neumann:}& \frac{\partial \phi}{\partial \nu} = 0 ~\text{on}~ \Sigma := (0,T) \times \partial \Omega.
\end{array}
\end{equation}
\end{enumerate}
\end{definition}
A special remark is in order concerning the integrability condition $(1+m)|\nabla_p H(m,\nabla u)|^2 \in L^1(Q)$.
The Fokker-Planck equation in \eqref{mfg} would still have a meaning in the sense of distributions given the weaker condition $m\nabla_p H(m,\nabla u) \in L^1(Q)$.
However, the stronger condition allows us to prove that weak solutions enjoy more regularity than given by the definition.
See Proposition \ref{fokker planck bounds weak}.
Such extra regularity makes possible our proof of uniqueness given in Section \ref{uniqueness section}.

We should also note that $|\nabla_p H(m,\nabla u)|^2 \in L^1(Q)$ does not follow from $H(m,\nabla u) \in L^1(Q)$ by the assumptions on $H$.
This is in contrast with the case where $H$ does not depend on $m$.
Indeed, if we removed the $m$ dependence from \eqref{DH bounded}, then $H(m,\nabla u) \in L^1(Q)$ would imply $\nabla_p H(m,\nabla u) \in L^{r/(r-1)}(Q) \subset L^2(Q)$.
In actual fact, \eqref{DH bounded} does not allow us to deduce $\nabla_p H(m,\nabla u) \in  L^2(Q)$ unless we already know that $m$ and $m^{-1}$ are bounded.

\subsection{Main results}

The following two theorems constitute the main results of this paper.
\begin{theorem}[Existence] \label{existence of weak solutions}
Assume the hypotheses given in Section \ref{assumptions}.
There exists a constant $T_0 > 0$ such that if $T \leq T_0$, then there exists a weak solution $(u,m)$ of \eqref{mfg}, defined in Section \ref{definition of solutions} below, such that
\begin{enumerate}
\item $u,m \in C([0,T];L^1(\Omega))$ with $u(T) = g(m(T))$ and $m(0) = m_0$, $m \geq 0$;
\item $m \in L^2(0,T;H^1(\Omega)) \cap L^\infty(Q),$ $1/m \in L^\infty(Q)$, and $\partial_t m \in L^r(0,T;W^{-1,r}(\Omega))$;
\item $u$ is bounded below.
\end{enumerate}
Furthermore, suppose that $f_L$ and $g_L$, defined in \ref{fL and gL defined}, are bounded for all $L > 0$.
Then $u$ is bounded and locally H\"older continuous in $Q$.
\end{theorem}

In order to formulate our uniqueness result, we will need an additional hypothesis on the structure of the Hamiltonian.
\begin{assumption} \label{uniqueness assumption}
We assume, in addition to the assumptions in Section \ref{assumptions}, that $H$ is $C^2$ in $p$ and that for all $m,\tilde{m} > 0$ and $p,\tilde{p} \in \bb{R}^d$,
\begin{equation} \label{uniqueness hamiltonian}
-\partial_m H(t,x,m,p)\tilde{m}^2 + m\nabla_p^2 H(t,x,m,p)(\tilde{p},\tilde{p})
 + m \tilde{m}\tilde{p}\cdot \partial_m \nabla_p H(t,x,m,p) > 0.
\end{equation}
\end{assumption}
A more concrete structure satisfying this assumption is given by
\begin{equation} \label{simple hamiltonian}
H(t,x,m,p) = h(t,x,m)|p|^r
\end{equation}
for some function $h(t,x,m)$ which is measurable in $(t,x)$, continuously differentiable in $m$, and satisfies
\begin{equation} \label{h growth}
0 \leq -\partial_m h(t,x,m)m < \frac{4(r-1)}{r}h(t,x,m) ~~~~ \forall m \geq 0.
\end{equation}
Note that if we take the canonical model from the introduction, then \eqref{h growth} implies
$$
h(t,x,m) = m^{-\alpha} ~~ \Rightarrow ~~ 0\leq \alpha < 4(r-1)/r.
$$

\begin{theorem}[Uniqueness]
\label{uniqueness}
In addition to the hypotheses given in Section \ref{assumptions}, suppose that $H$ satisfies Assumption \ref{uniqueness assumption}.
Then the weak solution provided by Theorem \ref{existence of weak solutions} is unique.
\end{theorem}

\section{Existence}
\label{existence section}

The proof of Theorem \ref{existence of weak solutions} is based on a priori estimates and compactness.
This section is divided into three parts, the first two of which address a priori estimates.
First, in Proposition \ref{fokker planck bounds}, we prove that solutions of the Fokker-Planck equation are a priori bounded both from above and away from zero, given a sufficient integrability condition on the vector field.
The other a priori estimates are standard for all mean field games.
Combined with Proposition \ref{fokker planck bounds} and known compactness results, we prove the existence of solutions to \eqref{mfg}.

\subsection{A priori estimates on the Fokker-Planck equation}

Consider the equation
\begin{equation} \label{fokker planck}
\left\{
\begin{array}{ll}
\partial_t m - \Delta m - \nabla \cdot (bm) = 0 & \text{in}~Q, \\
m(0) = m_0 & \text{in}~\Omega
\end{array}\right.
\end{equation}
with one of the following boundary conditions:
\begin{equation}
\label{boundary conditions fokker planck}
\begin{array}{ll}
\text{periodic:}& \Omega = \bb{T}^d, ~~ b,m ~\text{are}~\bb{Z}^d-\text{periodic in}~x,\\
\text{Dirichlet:}& m = 0 ~\text{on}~ \Sigma := (0,T) \times \partial \Omega, ~\text{or}\\
\text{Neumann:}& \frac{\partial m}{\partial \nu} + mb \cdot \nu = 0 ~\text{on}~ \Sigma := (0,T) \times \partial \Omega
\end{array}
\end{equation}
where $\nu$ is the outward unit normal vector on $\partial \Omega$.
If $b \in L^2(Q)$, we say that $m \in L^2(0,T;H^1(\Omega)) \cap L^\infty(Q)_+$ is a solution of \eqref{fokker planck}-\eqref{boundary conditions fokker planck} provided that
$$
\int_Q -m\partial_t \phi + \nabla m \cdot \nabla \phi + m b \cdot \nabla \phi = \int_{\Omega} m_0 \phi(0)
$$
for all test functions $\phi$ satisfying the  boundary conditions from \eqref{test function boundary} corresponding to \eqref{boundary conditions fokker planck}.

\begin{proposition} \label{fokker planck bounds}
Let $m_0 \in C(\Omega)$ be such that $m_0 > 0$, and let $b \in L^{2}(Q)$.
Suppose $m \in L^2(0,T;H^1(\Omega)) \cap L^\infty(Q)_+$ is a solution of \eqref{fokker planck}-\eqref{boundary conditions fokker planck}.
Suppose further that there exist non-negative measurable functions $b_1,\ldots,b_N$ and real constants $\beta_1,\ldots,\beta_N$ such that
\begin{equation} \label{b estimate}
|b| \leq \sum_{i=1}^N b_i, ~~~ m^{\beta_k}b_k \in L^{r/(r-1)} ~~ \forall ~k =1,\ldots,N. 
\end{equation}
Assume also that \eqref{growth rate of Hamiltonian} is satisfied.
Then there exist constants $T_0 > 0$,
$$
C(\|m_0\|_\infty,\{\|m^{\beta_k} b_k\|_{r/(r-1)}\}_{k=1}^N) 
~~
\text{and}
~~
C(\|m_0^{-1}\|_\infty,\{\|m^{\beta_k} b_k\|_{r/(r-1)}\}_{k=1}^N)
$$
such that if $T \leq T_0$, then
\begin{align} 
\label{m infinity esimate}
\sup_{t \in [0,T]}\left\|m(t)\right\|_\infty &\leq C(\|m_0\|_\infty,\{\|m^{\beta_k} b_k\|_{r/(r-1)}\}_{k=1}^N), \\
\label{1 over m estimate}
\sup_{t \in [0,T]}\left\|\frac{1}{m(t)} \right\|_\infty &\leq C(\|m_0^{-1}\|_\infty,\{\|m^{\beta_k} b_k\|_{r/(r-1)}\}_{k=1}^N).
\end{align}
Moreover, from \eqref{m infinity esimate} and \eqref{1 over m estimate} we deduce
\begin{equation}
\label{H^1 estimate}
\|\partial_t m\|_{L^{2}(0,T;H^{-1}(\Omega))} + \|m\|_{L^2(0,T;H^1(\Omega))} \leq C(\|m_0\|_\infty,\{\|m^{\beta_k} b_k\|_{r/(r-1)}\}_{k=1}^N).
\end{equation}
\end{proposition}

\begin{proof}
First we will prove \eqref{1 over m estimate}. Assume that $m$ is strictly positive; we will remove this assumption at the end of the proof.
Multiply \eqref{fokker planck} by $m^{-q-1}$ and integrate by parts, then use Young's inequality to estimate
\begin{equation} \label{first m estimate}
\int_\Omega m(t)^{-q} + \frac{q(q+1)}{2}\int_0^t \int_\Omega m^{-q-2}|\nabla m|^2 \leq \int_\Omega m_0^{-q} + \frac{q(q+1)}{2}\int_0^t \int_\Omega m^{-q}|b|^2.
\end{equation}
By the Sobolev embedding theorem we have
\begin{multline}
\int_0^T\|m^{-1}(t)\|_{dq/(d-2)}^q dt \leq C\int_0^T\|\nabla (m ^{-q/2}(t))\|_{2}^2 + \|m ^{-q/2}(t)\|_{2}^2 dt\\
 = C\frac{q^2}{4}\int_0^T \int_\Omega m ^{-q-2}|\nabla m|^2 ~dx dt + C\int_0^T \|m ^{-1}(t)\|_{q}^q dt.
\end{multline}
With the previous estimate this implies
\begin{multline} \label{holder part of moser}
\|m ^{-1}\|_{L^\infty_t L^q_x}^q + \|m ^{-1}\|_{L^q_tL^{dq/(d-2)}_x}^q \leq C\|m_0^{-1}\|_\infty^{q} + Cq^2\int_0^T \int_\Omega m ^{-q}|b|^{2}\\
\leq C\|m_0^{-1}\|_\infty^{q} + CNq^2\sum_{i=1}^N\int_0^T \int_\Omega m ^{-q}b_i^{2}\\
\leq C\|m_0^{-1}\|_\infty^{q} + CNq^2\sum_{i=1}^N\|m^{\beta_i} b_i\|_{r/(r-1)}^2\|m ^{-1}\|_{r^*(q+\gamma_i)}^{q+\gamma_i}
\end{multline}
where $r^* := r/(2-r)$, $\gamma_i := \max\{2\beta_i(r-1)/r,0\}$, and where $C$ is some constant depending only on the domain $\Omega$ and the final time $T$.

Noting that
\begin{equation}
\|m^{-1}\|_{q(1+2/d)} \leq \|m^{-1}\|_{L^\infty_t L^q_x}^{2/(d+2)}\|m^{-1}\|_{L^q_tL^{dq/(d-2)}_x}^{d/(d+2)},
\end{equation}
we can deduce
\begin{equation} \label{Moser iterative formula}
\|m^{-1}\|_{q(1+2/d)} \leq K^{1/q}q^{2/q}\max\{1,\|m_0^{-1}\|_\infty,\|m^{-1}\|_{r^*(q+\gamma)}\}^{1+\gamma/q},
\end{equation}
where $\gamma := \max\{\gamma_1,\ldots,\gamma_N\}$ and $K$ depends only on $\{\|m^{\beta_k} b_k\|_{r/(r-1)}\}_{k=1}^N$ and constant parameters.

In order to use the Moser iteration method, we need to show that for some $q$ large enough, $\|m^{-1}\|_q$ can be bounded by a constant depending only on the parameters as do $C$ and $K$.
Let us observe that \eqref{first m estimate} can be used to deduce
\begin{equation} \label{first iteration}
\int_\Omega m^{-q}(t) dx + \int_0^t \|m^{-1}(s)\|_{L_x^{dq/(d-2)}}^q ds \leq C + \int_0^t \int_{\Omega} m^{-r^*(q+\gamma)} dx ds,
\end{equation}
where again $C$ depends only on the parameters.
We wish to bound the right-hand side.
Choose $\lambda \in (0,1)$ such that
$$
q\lambda + (1-\lambda)dq/(d-2) = (q+\gamma)r^*,
$$
that is,
$$
\lambda = \frac{dq - (d-2)r^*(q+\gamma)}{2q}.
$$
We also set
$$
\mu = \frac{d}{d-2}(1-\lambda) = \frac{d(r^*-1+\gamma r^*/q)}{2}.
$$
Note that $\mu \to \frac{d}{2}(r^*-1) < 1$ as $q \to \infty$, hence for $q$ large enough we have $\mu < 1$ as well.
So we estimate
\begin{multline}
\int_{\Omega} m^{-r^*(q+\gamma)}(t) dx \leq \left(\int_{\Omega} m^{-q}(t) dx\right)^\lambda \left(\int_{\Omega} m^{-dq/(d-2)}(t) dx\right)^{1-\lambda} \\
= \left(\int_{\Omega} m^{q}(t) dx\right)^\lambda \|m^{-1}(t)\|_{dq/(d-2)}^{q\mu}
\leq C \left(\int_{\Omega} m^{q}(t) dx\right)^{\lambda/(1-\mu)} +  \|m^{-1}(t)\|_{dq/(d-2)}^{q}
\end{multline}
where in this line $C$ depends only on $\mu$.
Plugging this into \eqref{first iteration}, we see that
\begin{equation}
\int_{\Omega} m^{-q}(t)dx \leq C + C\int_0^t \left(\int_{\Omega} m^{q}(s) dx\right)^{\lambda/(1-\mu)} ds.
\end{equation}
Set $\beta = \lambda/(1-\mu)$ and $H(t) = \int_0^t \left(\int_{\Omega} m^{q}(s) dx\right)^{\lambda/(1-\mu)} ds$.
This implies the ordinary differential inequality
$$
\dot{H}(t) \leq C(1+H(t))^\beta, ~~ H(0) = 0.
$$
In particular, we have
$$
H(t) \leq \frac{1}{1-C(t_*-t)^{\beta-1}} ~~ \forall ~ t < t_* := \frac{1}{C(\beta-1)}.
$$
Hence if $T$ is small enough, we have that $H(t)$ and thus also $\|m^{-1}(t)\|_q$ is bounded by a constant for all $t \in [0,T]$.
This will allow us to start the Moser iteration.

We now define sequences $\{q_n\}$ and $\{z_n\}$ as follows.
Set $s := (1+2/d)/r^*$; by the assumption $r < 1 + (d+1)^{-1}$ we have $s > 1$. 
Fix some
\begin{equation}
1 < u < s ~\text{and}~ q_0 > \frac{\gamma}{s-u}
\end{equation}
where $q_0$ is large enough so that $\|m^{-1}\|_{r^*(q_0+\gamma)} \leq C$.
Now set
\begin{equation}
q_{n+1} = q_ns - \gamma, ~~ z_n = \max\{1,\|m_0^{-1}\|_{\infty},\|m^{-1}\|_{r^*(q_n+\gamma)}\}, ~~ \forall ~n \geq 0.
\end{equation}
Note that $\{q_n\}$ grows exponentially:
\begin{equation}
q_0u^n \leq q_n \leq q_0 s^n.
\end{equation}
Apply \eqref{Moser iterative formula} to get
\begin{equation}
z_{n+1} \leq K^{1/q_n}q_n^{2/q_n}z_n^{1+\gamma/q_n}.
\end{equation}
Using the bounds on $q_n$ this implies
\begin{equation} \label{inductive step}
\log(z_{n+1}) \leq \frac{\log(Kq_0^2s^{2n})}{q_0u^n} + \left(1+\frac{\gamma}{q_0u^n}\right)\log(z_n)
\leq \frac{A(n+1)}{u^n} + \left(1+\frac{B}{u^n}\right)\log(z_n)
\end{equation}
where
$$
A = \frac{\max\{\log(Kq_0^2),\log(s^{2})\}}{q_0}, ~~ B = \frac{\gamma}{q_0}.
$$
By induction, one deduces from \eqref{inductive step} that
\begin{equation} \label{log zn estimate}
\log(z_n) \leq A\left(\sum_{k=0}^{n-1} \frac{k+1}{r^k} \right) \prod_{k=0}^{n-1}\left(1+\frac{B}{r^k}\right)\max\{\log(z_0),1\},
\end{equation}
which converges because $r > 1$.
Setting $z_\infty := \limsup_{n\to \infty} z_n$ we can conclude that
\begin{equation}
\|m^{-1}\|_\infty \leq z_\infty < +\infty
\end{equation}
where $z_\infty$ is a constant depending only on $\Omega,T$, $\|m_0^{-1}\|_\infty$ and $\|m^{\beta_k} b_k\|_{r/(r-1)}\}_{k=1}^N$.
Appealing once more to \eqref{first m estimate} we get more precisely the estimate \eqref{1 over m estimate}.

In the proof above, in order to remove that assumption that $m$ is strictly positive, simply replace $m$ in with $m+\epsilon$ for $\epsilon > 0$ arbitrary.
Then let $\epsilon \to 0$ to get the result.

To prove \eqref{m infinity esimate}, we use an analogous argument.
In particular, multiplying \eqref{fokker planck} by $m^{q-1}$ we arrive at
\begin{equation} \label{main fokker planck estimate}
\frac{1}{q}\int_{\Omega} m(t)^q + \frac{q-1}{2}\int_0^t \int_{\Omega} m^{q-2}|\nabla m|^2 \leq \frac{1}{q}\int_{\Omega} m_0^{q} + \frac{q+1}{2}\int_0^t \int_\Omega m^q |b|^2.
\end{equation}
By the Sobolev embedding theorem and some calculations (cf. \eqref{holder part of moser}) we obtain
\begin{equation}
\|m\|_{L_t^\infty L_x^q}^q + \|m\|_{L_t^q L_x^{dq/(d-2)}}^q \leq C\|m_0\|_\infty^q + CNq^2\sum_{i=1}^{N}\|m^{\beta_i}b_i\|_{r/(r-1)}^2 \|m\|_{r^*(q+\gamma_i)}^{q+\gamma_i},
\end{equation}
where in this case $\gamma_i = \max\{-2\beta_i(r-1)/r,0\}$.
This leads, using the same argument as above, to
\begin{equation}
\|m\|_{q(1+2/d)} \leq K^{1/q}q^{2/q}\max\{1,\|m_0\|_\infty,\|m\|_{r^*(q+\gamma)}\}^{1+\gamma/q}
\end{equation}
where as before $\gamma = \max\{\gamma_1,\ldots,\gamma_N\}$.
Then a similar iteration procedure as before leads to \eqref{m infinity esimate}.

Finally, we deduce \eqref{H^1 estimate} from \eqref{m infinity esimate} and \eqref{1 over m estimate} as follows.
Plug in the estimate \eqref{m infinity esimate} into \eqref{main fokker planck estimate} with $q =2$ to see that $m$ is bounded in $L^2(0,T;H^1(\Omega))$.
Then observe that \eqref{b estimate} together with \eqref{m infinity esimate} and \eqref{1 over m estimate} imply that $mb \in L^{r/(r-1)}$.
We deduce from the Fokker-Planck equation that $\partial_t m$ is bounded in $L^2(0,T;H^{-1}(\Omega))$.
\end{proof}

\subsection{More a priori estimates}

Here we collect some standard energy  estimates for the mean field game system.

\begin{lemma}[Main energy estimate] \label{main energy estimate}
Let $f=f(t,x,s), g=g(x,s),$ and $H = H(t,x,s,\xi)$ be given, and suppose $(u,m)$ is a bounded solution of \eqref{mfg}.
Assume $f,g,$ and $H$ are bounded below by some constant $-C$.
Then, suppressing the dependence on $(t,x)$, we have
\begin{multline} \label{Main energy estimate}
\int_\Omega g(m(T))m(T) + \int_0^T \int_\Omega f(m)m + (\nabla u \cdot \nabla_p H(m,\nabla u) - H(m,\nabla u))m + \|m_0\|_\infty \int_0^T \int_\Omega H(m,\nabla u)\\ \leq C\|m_0\|_\infty \left(\int_\Omega g_{2\|m_0\|_\infty} + \int_0^T \int_\Omega f_{2\|m_0\|_\infty} + C\right)
\end{multline}
where $f_{2\|m_0\|_\infty}$ and $g_{2\|m_0\|_\infty}$ are defined in \eqref{fL and gL defined}.

In particular, if $m_0 \in L^\infty$ and $f_{2\|m_0\|_\infty}$ and $g_{2\|m_0\|_\infty}$ are integrable, then
\begin{equation} \label{key terms}
\|f(m)m\|_{L^1_{t,x}} + \|g(m(T))m(T)\|_{L^1_x} + \|m(\nabla u \cdot \nabla_p H(m,\nabla u) - H(m,\nabla u))\|_{L^1_{t,x}} + \|H(m,\nabla u)\|_{L^1_{t,x}} \leq C.
\end{equation}
\end{lemma}

\begin{proof}
We rely on the rather simple but important observation that
\begin{equation} \label{f and fL}
f(m) \leq f_{L} + \frac{1}{L} f(m)m + Cm, ~~ g(m) \leq g_{L} + \frac{1}{L} g(m)m + Cm
\end{equation}
We will apply this below with $L = 2\|m_0\|_\infty$.

Start with the energy identity,
\begin{equation} \label{energy identity}
\int_\Omega g(m(T))m(T) + \int_0^T \int_\Omega f(m)m + (\nabla u \cdot \nabla_p H(m,\nabla u) - H(m,\nabla u))m = \int_\Omega u(0)m_0,
\end{equation}
which is obtained by adding $m$ times the Hamilton-Jacobi equation to $u$ times the Fokker-Planck equation and then integrating by parts.
Then using the lower bound on $u$,
\begin{multline}
\int_\Omega u(0)m_0 \leq \int_\Omega (u(0)+C)m_0 \leq \|m_0\|_\infty \int_\Omega (u(0)+C)
\leq \|m_0\|_\infty \left(\int_\Omega u(T) - \int_0^T \int_\Omega \partial_t u + C\right)\\
\leq \|m_0\|_\infty \left(\int_\Omega g(m(T)) + \int_0^T \int_\Omega f(m) - \int_0^T \int_\Omega H(m,\nabla u) + C\right).
\end{multline}
Use \eqref{f and fL} and plug into \eqref{energy identity} to get \eqref{Main energy estimate}.

As for \eqref{key terms}, it suffices to observe that each of the terms in \eqref{Main energy estimate} is non-negative: $f(m)m$ and $g(m(T))m(T)$ because $f$ and $g$ are monotone, $\nabla u \cdot \nabla_p H - H$ because $H$ is convex in $p$, and $H$ by hypothesis \eqref{H positive}.
\end{proof}

\begin{corollary} \label{integrability of solutions}
Supose $(u,m)$ is a smooth solution of \eqref{mfg}.
The following estimates hold for a small enough time $T$ and for a constant $C$ depending only on the data.
\begin{enumerate}
\item Estimates \eqref{m infinity esimate},  \eqref{1 over m estimate}, and \eqref{H^1 estimate} hold.
\item $\|\nabla_p H(m,\nabla u)\|_{r/(r-1)} \leq C$.
\item $\|f(m)\|_{L^1_{t,x}} + \|g(m(T))\|_{L^1_x} \leq C$.
\item $\|\partial_t u + \Delta u\|_{L^1_{t,x}} \leq C$.
\end{enumerate}

\end{corollary}

\begin{proof}

1. As a result of \eqref{DH bounded}, \eqref{H positive}, and then \eqref{DH bounded below}, we have
\begin{multline} \label{gradient H estimate}
|\nabla_p H(t,x,m,p)|^{r/(r-1)} \leq C(m^{\lambda}+m^{-\lambda})^{r/(r-1)}(1+|p|^{r}) \\
\leq C(m^{\lambda }+m^{-\lambda })^{2r/(r-1)}(1+H(t,x,m,p)) \\
\leq C(m^{\lambda}+m^{-\lambda})^{2r/(r-1)}(1+\nabla_p H(t,x,m,p) \cdot p - H(t,x,m,p)).
\end{multline}
Let 
$
b = \nabla_p H(t,x,m,\nabla u)
$
and
$
v = (1+\nabla_p H(t,x,m,p) \cdot p - H(t,x,m,p))^{(r-1)/r}.
$
Then we have
$$
|b| \leq Cm^{2\lambda} + Cm^{-2\lambda} + Cm^{2\lambda}v + Cm^{-2\lambda}v =: b_1 + b_2 + b_3 + b_4.
$$
Moreover, from \eqref{key terms} and the definition of $v$ it follows that $\|m|v|^{r/(r-1)}\|_1 \leq C$.
Therefore we can choose some $\beta_i$, $i=1,2,3,4$, such that
$
\|m^{\beta_i}b_i\|_{r/(r-1)} \leq C.
$
Hence we reach the desired conclusion by Proposition \ref{fokker planck bounds}.

2. Combining the estimate $\|m\|_\infty + \|m^{-1}\|_\infty \leq C$ from the previous step with \eqref{gradient H estimate} and \eqref{key terms}, we conclude that $\|\nabla_p H(m,\nabla u)\|_{r/(r-1)} \leq C$.

3. This follows from \eqref{key terms} together with \eqref{f and fL}.

4. This follows by using the estimate $\|H(m,\nabla u)\|_1 \leq C$ from \eqref{key terms} together with part 4 in the Hamilton-Jacobi equation.
\end{proof}

\begin{remark}
We observe that hypotheses \eqref{H positive},\eqref{DH bounded}, and \eqref{DH bounded below} were not used directly to prove the a priori estimates in Corollary \ref{integrability of solutions}, but only to derive the estimate \eqref{gradient H estimate}.
This will be useful in the following section, where we show that a particular approximation of $H$ which is Lipschitz in $p$ satisfies \eqref{gradient H estimate}, even though it does not satisfy \eqref{H positive},\eqref{DH bounded}, and \eqref{DH bounded below}.
\end{remark}

\subsection{Proof of Theorem \ref{existence of weak solutions}}

In order to construct solutions of \eqref{mfg}, we first derive smooth solutions to an approximation of \eqref{mfg}, and by compactness we show that they converge to weak solutions of \eqref{mfg}.
For the first step, we have
\begin{lemma} \label{bounded solutions}
Let $H(t,x,m,p)$ satisfy the following:
$$
|H(t,x,m,0)| + |\nabla_p H(t,x,m,p)| \leq C.
$$
Suppose $f,g$ are also bounded.
Then there exists a bounded (smooth in $(0,T)$) solution $(u,m)$ to \eqref{mfg}.
\end{lemma}

\begin{proof}
Set
$$
X_L = \{m \in C([0,T];L^2(\Omega)) \cap L^\infty(Q ) : \|m\|_\infty \leq L\}.
$$
For any $\mu \in X_L$, define $u_\mu \in L^2(0,T;H^1(\Omega)) \cap L^\infty(Q )$ to be the (unique) bounded solution of
\begin{equation}
\left\{
\begin{array}{l}
-\partial_t u - \Delta u + H(t,x,m,\nabla u) = f(t,x,\mu)  \\
u(T,x) = g(x,\mu(T,x))
\end{array}\right.
\end{equation}
We then set $m = \Phi(\mu)$ to be the solution of
\begin{equation}
\left\{
\begin{array}{l}
\partial_t m - \Delta m - \mathrm{div} \left(m\nabla_p H(t,x,m,\nabla u_\mu)\right) = 0  \\
m(0,x) = m_0(x)
\end{array}\right.
\end{equation}
By standard results on parabolic equations there exists a constant $L$ depending on $\|\nabla_p H\|_\infty,\|m_0\|_\infty$ such that $m$ is bounded by $L$ uniformly in $\mu$.
For this $L$, it follows that $X_L$ is an invariant convex subset of $C([0,T];L^2(\Omega))$ under the map $\Phi$, which is continuous and compact.
By Schauder's fixed point theorem, we conclude.
\end{proof}

We now define an approximation of system \eqref{mfg} satisfying the hypotheses of Lemma \ref{bounded solutions}.
For $\epsilon > 0$, let
\begin{equation}
H_\epsilon(t,x,m,p) = \frac{H(t,x,\kappa_\epsilon(m),p)}{1+\epsilon H(t,x,\kappa_\epsilon(m),p)^{(r-1)/r}}
\end{equation}
where
$$
\kappa_\epsilon(m) := \min\left\{\max\{m,\epsilon\},\frac{1}{\epsilon}\right\}.
$$
It follows that
\begin{equation}
\nabla_p H_\epsilon(t,x,m,p) = \frac{(1+\frac{\epsilon}{r}H(t,x,\kappa_\epsilon(m),p)^{(r-1)/r})\nabla_p H(t,x,\kappa_\epsilon(m),p)}{(1+\epsilon H(t,x,\kappa_\epsilon(m),p)^{(r-1)/r})^2}
\end{equation}
so that by \eqref{H positive} and \eqref{DH bounded} we have
$$
|\nabla_p H_\epsilon(t,x,m,p)| \leq \frac{|\nabla_p H(t,x,\kappa_\epsilon(m),p)|}{1 + \epsilon H(t,x,\kappa_\epsilon(m),p)^{(r-1)/r}}
\leq \frac{C}{\epsilon}(\kappa_\epsilon(m)^{2\lambda}+\kappa_\epsilon(m)^{-2\lambda})
\leq C(\epsilon).
$$
We also set $f_\epsilon := \min\{f,\frac{1}{\epsilon}\},g_\epsilon := \min\{g,\frac{1}{\epsilon}\}$.
Note that $H_{\epsilon}(t,x,m,p),f_{\epsilon}(t,x,m),$ and $g_{\epsilon}(x,m)$ satisfy the hypotheses of Lemma \ref{bounded solutions}.

Additionally, we claim $H_\epsilon$ also satisfies a form of the estimate \eqref{gradient H estimate}, which together with Lemma \ref{main energy estimate} is sufficient to give the a priori bounds from Corollary \ref{integrability of solutions}.
From \eqref{DH bounded below} we see that
\begin{equation}
\nabla_p H_\epsilon(t,x,m,p) \cdot p -  H_\epsilon(t,x,m,p)
\geq \frac{(r-1)H(t,x,\kappa_\epsilon(m),p)}{(1+\epsilon H(t,x,\kappa_\epsilon(m),p)^{(r-1)/r})^2} - C,
\end{equation}
while from \eqref{H positive} and \eqref{DH bounded} we get
\begin{align} 
 \label{H epsilon first estimate}
|\nabla_p H_\epsilon(t,x,m,p)|^{r/(r-1)} 
& \leq C(\kappa_\epsilon(m)^\lambda + \kappa_\epsilon(m)^{-\lambda})^{2r/(r-1)}\frac{1+H(t,x,\kappa_\epsilon(m),p)}{(1+\epsilon H(t,x,\kappa_\epsilon(m),p)^{(r-1)/r})^2} 
\\ \nonumber
 & \leq C(\kappa_\epsilon(m)^\lambda + \kappa_\epsilon(m)^{-\lambda})^{2r/(r-1)}(1+H_\epsilon(t,x,m,p))
\end{align}
using the fact that $r/(r-1) > 2$.
We deduce
\begin{equation} \label{H epsilon structure}
|\nabla_p H_\epsilon(t,x,m,p)|^{r/(r-1)} \leq C(m^{2r\lambda /(r-1)} + m^{-2r\lambda /(r-1)})(1+\nabla_p H_\epsilon(t,x,m,p) \cdot p -  H_\epsilon(t,x,m,p)),
\end{equation}
which gives an estimate of the form \eqref{gradient H estimate} with an adjusted value of $\lambda$.
(Such an adjustment of the parameter $\lambda$ does not change the a priori estimates obtained from Corollary \ref{integrability of solutions}.)

Now let $(u_\epsilon,m_\epsilon)$ be bounded solutions of \eqref{mfg} with $H,f,g$ replaced by $H_\epsilon,f_\epsilon,g_\epsilon$, respectively.
We claim that, up to a subsequence, $(u_\epsilon,m_\epsilon)$ converges to a weak solution $(u,m)$ of \eqref{mfg}.
Recall that we are assuming $T$ is small enough so that Corollary \ref{integrability of solutions} is valid.

From Corollary \ref{integrability of solutions} we see that $-\partial_t u_\epsilon - \Delta u_\epsilon$ is bounded in $L^1(Q)$ and that $u_\epsilon(T)$ is bounded in $L^1(\Omega)$.
This implies that $u_\epsilon$ and $\nabla u_\epsilon$ are relatively compact in $L^1(Q)$, so for some $u \in L^1(Q)$ we have, up to a subsequence, $u_\epsilon \to u$ strongly in $L^1(Q)$, $\nabla u_\epsilon \to \nabla u_\epsilon$ strongly in $L^1(Q)$ as well as pointwise almost everywhere.

On the other hand, we have that $m_\epsilon$ satisfies \eqref{m infinity esimate}, \eqref{1 over m estimate}, and \eqref{H^1 estimate} uniformly in $\epsilon$, that is,
$$
\|m_\epsilon\|_\infty + \|m_\epsilon\|_{L^2(0,T;H^1(\Omega))} + \|\partial_t m_\epsilon\|_{L^r(0,T;W^{-1,r}(\Omega))} \leq C.
$$
From standard compactness results (cf. \cite{simon86}) we have for some subsequence that $m_\epsilon$ strongly converges in $L^2(Q)$ and almost everywhere on a subsequence  to some $m$.
Moreover, by part 2 of Corollary \ref{integrability of solutions} we have that $\|\nabla_p H_\epsilon(t,x,m_\epsilon,\nabla u_\epsilon)\|_{r/(r-1)} \leq C$, and since $r/(r-1) > 2$ we can use H\"older's inequality to see that $m_\epsilon|\nabla_p H_\epsilon(t,x,m_\epsilon,\nabla u_\epsilon)|^2$ is uniformly integrable.
We conclude that
$$
\sqrt{m_\epsilon}\nabla_p H_\epsilon(t,x,m_\epsilon,\nabla u_\epsilon) \to \sqrt{m}\nabla_p H(t,x,m,\nabla u) ~~\text{strongly in}~L^2(Q).
$$
Applying \cite[Theorem 6.1]{porretta2013weak} we obtain $m_\epsilon \to m$ in $C([0,T];L^1(\Omega))$ and that $m$ is a weak solution of the corresponding Fokker-Planck equation.

It remains to show that $u$ solves the Hamilton-Jacobi equation.
Since $m_\epsilon \to m$ and $\nabla u_\epsilon \to \nabla u$ almost everywhere, then by Fatou's Lemma we have
$$
\int_Q H(t,x,m,\nabla u)dx dt \leq \liminf_{\epsilon \to 0} \int_Q H_\epsilon(t,x,m_\epsilon,\nabla u_\epsilon)dx dt \leq C
$$
by \eqref{key terms}.
By hypothesis \eqref{H positive} together with the fact that $m$ and $m^{-1}$ are bounded, we see that $\|\nabla u\|_r \leq C$ so that $u \in L^r(0,T;W^{1,r}(\Omega))$.

Now it can be shown using \eqref{f and fL} that $g_\epsilon(\cdot,m_\epsilon(T))$ and $f_\epsilon(\cdot,\cdot,m_\epsilon)$ are equi-integrable.
Indeed, we have
$$
\int_{E} g_\epsilon(\cdot,m_\epsilon(T))
\leq \int_{E} g_L + \frac{1}{L}\int_{E} g_\epsilon(\cdot,m_\epsilon(T))m_\epsilon(T) + C\int_{E} m_\epsilon(T).
$$
Taking the measure of $E$ to zero and then $L$ to infinity we see that the right-hand side goes to zero, as desired, so $g_\epsilon(\cdot,m_\epsilon(T))$ is equi-integrable.
The argument for  $f(\cdot,\cdot,m_\epsilon)$ is analogous.
Given that $m_\epsilon(T) \to m(T)$ and $m_\epsilon \to m$ almost everywhere, it follows that
$$
g_\epsilon(\cdot,m_\epsilon(T)) \to g(\cdot,m(T)) ~~\text{in}~L^1(\Omega),
~~ f_\epsilon(\cdot,\cdot,m_\epsilon) \to f(\cdot,\cdot,m)~~\text{in}~L^1(Q).
$$
As for the Hamiltonian term, because we have that $m_\epsilon$ and $m_\epsilon^{-1}$ are uniformly bounded, $H_\epsilon(t,x,m_\epsilon,\nabla u_\epsilon)$ therefore has uniformly subquadratic growth in $\nabla u_\epsilon$ by hypothesis \eqref{DH bounded}.
We can now invoke previous results on parabolic equations (see for instance \cite{porretta1999existence} and references therein) to conclude that $H_\epsilon(t,x,m_\epsilon,\nabla u_\epsilon)$ strongly converges in $L^1(Q)$ to $H(t,x,m,\nabla u)$ and that $u$ is a weak solution to the Hamilton-Jacobi equation.

Finally, suppose we have that $f_L$ and $g_L$ are bounded for all $L > 0$.
Then by \eqref{m infinity esimate} it follows that $f(m)$ and $g(m(T))$ are bounded.
So by classical results on quasilinear parabolic equations, $u$ is bounded and locally H\"older continuous (see any text on parabolic equations, e.g. \cite{ladyzhenskaia1968linear}).

This completes the proof of Theorem \ref{existence of weak solutions}.

\begin{remark}
Naturally, we could show that the solution is even smooth, if we assume that $H$ is smooth enough.
\end{remark}

\section{Uniqueness}
\label{uniqueness section}

The proof of Theorem \ref{uniqueness} requires two main ingredients.
First, we prove that the estimates of Proposition \ref{fokker planck bounds} are valid for all weak solutions of the Fokker-Planck equation, and not just those with sufficient regularity.
Second, we formulate the structure conditions on the Hamiltonian which allow us to prove uniqueness.

\subsection{Regularity of weak solutions}

To obtain existence of solutions, it was sufficient to have a priori bounds on a sequence of smooth solutions, from which we showed that this sequence was compact.
On the other hand, in order to establish uniqueness of weak solutions without restricting their definition, it is necessary to show that they are regular enough to perform integration by parts.
For this we will prove an extension of Proposition \ref{fokker planck bounds}.

First, we give a weaker definition of solutions of the Fokker-Planck equation \eqref{fokker planck}-\eqref{boundary conditions fokker planck}.
Suppose $b \in L^2(Q)$.
\begin{definition}
[Weak solutions of the Fokker-Planck equation]
\label{weak solution fokker-planck}
A function $m \in L^1(Q)_+$ is a weak solution of \eqref{fokker planck}-\eqref{boundary conditions fokker planck} provided that
$$
m|b|^2 \in L^1(Q)
$$
and
$$
\int_Q m(-\phi_t -\Delta \phi + b \cdot \nabla \phi) = \int_{\Omega} m_0 \phi(0)
$$
for every $\phi \in C_c^\infty([0,T),\overline{\Omega})$ satisfying the corresponding boundary conditions taken from \eqref{test function boundary}.
\end{definition}
The above definition, taken from \cite{porretta2013weak}, has two interesting properties.
One is that, thanks to the condition that $m|b|^2 \in L^1(Q)$, one can show weak solutions defined in this way are unique.
The second, which is of more interest to us, is that weak solutions are equivalent to {\em renormalized solutions.}

Define, for $k > 0$, the truncation function
\begin{equation}
\label{truncation function}
T_k(s) := \min(k,\max(s,-k)).
\end{equation}
A renormalized solution of \eqref{fokker planck} is defined as a function $m \in L^1(Q)_+$ such that
\begin{equation} \label{renormalized fokker planck}
\begin{array}{cl}
(i) & T_k(m) \in L^2(0,T;H^1(\Omega)) ~~ \forall  k > 0. \\
(ii) & \displaystyle \lim_{N \to \infty} \frac{1}{N} \iint_{N < |m| < 2N} |\nabla m |^2 = 0 \\
(iii) & \left\{\begin{array}{ll}
\partial_t S(m ) - \Delta S(m ) - \mathrm{div}(S'(m)m b) = - S''(m )|\nabla m |^2 - S''(m)m b \cdot \nabla m & \text{in}~Q \\
S(m )(0) = S(m _0) & \text{in}~\Omega
\end{array}\right.
\end{array}
\end{equation}
for every $S \in W^{2,\infty}(\bb{R})$ such that $S'$ has compact support.
Part (iii) is understood in the sense of distributions, again choosing test functions according to the specified boundary conditions.

The equivalence between weak and renormalized solutions of \eqref{fokker planck} will be crucial in the proof of the following:

\begin{proposition}
\label{fokker planck bounds weak}
Assume the same hypotheses as in Proposition \ref{fokker planck bounds}, except that $m \in L^1(Q)_+$ is a weak solution of \eqref{fokker planck}-\eqref{boundary conditions fokker planck}.
Then the conclusion of Proposition \ref{fokker planck bounds} holds.
\end{proposition}

\begin{proof}
By \cite[Theorem 3.6]{porretta2013weak}, $m \in C([0,T];L^1(\Omega))$ and is a renormalized solution.
Let $\epsilon > 0$ be small, $N > 0$ be large.
Consider the auxiliary function
\begin{equation}
S(m) :=
\left\{
\begin{array}{ll}
\epsilon^{-q} &\text{if}~ m \leq 0\\
(m+\epsilon)^{-q} &\text{if}~0 \leq m \leq N\\
(N+\epsilon)^{-q} + \frac{q}{2}(N+\epsilon)^{-q-1}[(m-N-1)^2-1] &\text{if}~ N \leq m \leq N+1\\
(N+\epsilon)^{-q} - \frac{q}{2}(N+\epsilon)^{-q-1} &\text{if}~ N + 1 \leq m
\end{array}\right.
\end{equation}
which satisfies $S \in W^{2,\infty}(\bb{R})$ and $S'$ has compact support. 
Now integrating \eqref{renormalized fokker planck} and using the fact that $S'' \geq 0$ we get
\begin{equation} \label{renormalized estimate}
\int_\Omega S(m(t)) + \frac{1}{2}\int_0^t \int_\Omega S''(m)|\nabla m|^2 \leq \int_\Omega S(m_0) + \frac{1}{2}\int_0^t \int_\Omega S''(m)m^{2}|b|^2.
\end{equation}
This implies
\begin{multline}
\frac{1}{q}\int_{m(t) \leq N} (m(t)+\epsilon)^{-q} + \frac{q+1}{2}\iint_{m \leq N} (m+\epsilon)^{-q-2}|\nabla m|^2 \\
\leq \frac{1}{q}\int_\Omega m_0^{-q} + |\Omega|(N+\epsilon)^{-q} +  \frac{q+1}{2}\iint_{m \leq N} (m+\epsilon)^{-q-2\alpha}|b|^2 + \frac{q}{2}(N+\epsilon)^{-q-1}(N+1)^{2}\int_Q |b|^2.
\end{multline}
We may assume $q$ is large. Let $N \to \infty$ to obtain
\begin{equation}
\frac{1}{q}\int_\Omega (m(t)+\epsilon)^{-q} + \frac{q+1}{2}\int_0^t \int_\Omega (m+\epsilon)^{-q-2}|\nabla m|^2 \leq \frac{1}{q}\int_\Omega m_0^{-q} + \frac{q+1}{2}\int_0^t \int_\Omega (m+\epsilon)^{-q}|b|^2.
\end{equation}
Then we may follow the proof of Proposition \ref{fokker planck bounds} to see that \eqref{1 over m estimate} holds.
The proof of \eqref{m infinity esimate} is analogous (replace $-q$ with $q$).

Finally, to see that \eqref{H^1 estimate} holds, note that we have $0 \leq m \leq N$ almost everywhere for some $N > 0$.
Define a new auxiliary function $S \in W^{2,\infty}(\bb{R})$ such that $S'$ has compact support, and such that $S(r) = \frac{1}{2}r^2$ for $r \in [0,N]$.
Then \eqref{renormalized estimate} implies
\begin{equation}
\frac{1}{2}\int_{\Omega} m(t)^2 + \frac{1}{2}\iint_{Q} |\nabla m|^2 \leq \frac{1}{2}\int_\Omega S(m_0) + \frac{\|m\|_\infty}{2}\int_0^t \int_\Omega |b|^2.
\end{equation}
We conclude that $m \in L^2(0,T;H^1(\Omega))$, from which \eqref{H^1 estimate} follows.
\end{proof}

\begin{remark}
\label{use of proposition}
By Definition \ref{definition of solutions}, the result we have just proved implies that the conclusion of Proposition \ref{fokker planck bounds} applies to $m$ for any weak solution $(u,m)$ of \eqref{mfg}.
\end{remark}

\begin{remark}
\label{difficulty of approximation}
One may consider proving Proposition \ref{fokker planck bounds weak} by approximating weak solutions of the Fokker-Planck equation with smooth ones, appealing to Proposition \ref{fokker planck bounds} and then passing to the limit.
The difficulty one then encounters is to show that one can find such a smooth approximation such that \eqref{b estimate} also holds.
For instance, if one takes an arbitrary smooth approximation $b_\epsilon$ of the vector field $b$, then one risks losing the estimate \eqref{b estimate}.
It turns out that using the theory of renormalized solutions is a much more tractable way of extending the essential estimates to weak solutions.
\end{remark}

\subsection{Proof of Theorem \ref{uniqueness}}

Now that we know weak solutions have some extra regularity, it will be possible to prove uniqueness using integration by parts.
Before proceeding, we need a technical lemma.
Recall that Assumption \eqref{uniqueness hamiltonian} is in place, which allows us to prove the following:
\begin{lemma} \label{convexity}
Assume that \eqref{uniqueness hamiltonian} holds.
Then for all $m_1,m_0 > 0$ and all $p_1,p_0 \in \bb{R}^d$, we have
\begin{equation}
 (m_1\nabla_p H(m_1,p_1)-m_0\nabla_p H(m_0,p_0)) \cdot (p_1-p_0)
  -(H(m_1,p_1)-H(m_0,p_0))(m_1-m_0) \geq 0
\end{equation}
where strict inequality holds if $p_1 \neq p_0$.
\end{lemma}
\begin{proof}
For $\theta \in [0,1]$, set $p_\theta = p_0 + \theta(p_1-p_0)$ and $m_\theta = m_0 + \theta(m_1-m_0)$,
and let
\begin{equation*}
\phi(\theta) = (m_\theta\nabla_p H(m_\theta,p_\theta)-m_0\nabla_p H(m_0,p_0)) \cdot (p_1-p_0)
  -(H(m_\theta,p_\theta)-H(m_0,p_0))(m_1-m_0).
\end{equation*}
We compute
\begin{multline*}
\phi'(\theta) = -\partial_m H(m_\theta,p_\theta)(m_1-m_0)^2 + m_\theta\nabla_p^2 H(m_\theta,p_\theta)(p_1-p_0,p_1-p_0)
\\ + m_\theta(m_1-m_0)(p_1-p_0)\cdot \partial_m \nabla_p H(m_\theta,p_\theta).
\end{multline*}
Use \eqref{uniqueness hamiltonian} to see that $\phi'(\theta) \geq 0$ for all $\theta \in [0,1]$.
Since $\phi(0) = 0$, it follows that $\phi(1) \geq 0$, which is what we wanted to prove.
In the case where $p_1 \neq p_0$, it follows that $p_\theta = 0$ for at most one $\theta \in [0,1]$, so the above calculation shows $\phi'(\theta) > 0$ for all but at most one value of $\theta \in [0,1]$.
We conclude that $\phi(1) > 0$, as desired.
\end{proof}

Now we prove Theorem \ref{uniqueness}.

Let $(u_1,m_1)$ and $(u_2,m_2)$ be two solutions.
Introduce the following sequence of auxiliary functions:
\begin{equation} \label{S_N}
S_n(r) = nS_1\left(\frac{r}{n}\right), ~~ S_1(r) = \int_0^r S'_1(t)dt, ~~ S'_1(r) = \left\{\begin{array}{cl}
1 &\text{if}~|s|\leq 1,\\
2 - |s| &\text{if}~1< |s| \leq 2,\\
0 &\text{if}~|s| > 2.
\end{array}\right.
\end{equation}
Then for $i=1,2$ and for $k > 0$ we have that $S_k(u_1),S_k(u_2) \in L^2(0,T;H^1(\Omega)) \cap L^\infty(Q)$, and $u_i$ satisfies the renormalized equation
\begin{equation}
\label{renormalized hamilton jacobi}
\begin{array}{c}
\vspace{5pt}
-\partial_t S_k(u_i) - \Delta S_k(u_i) + S_k'(u_i)H(t,x,m_i,\nabla u_i) = S_k'(u_i)f(m_i) - S_k''(u_i)|\nabla u_i|^2,\\
S_k(u_i)(T) = S_k(g(m_i(T))).
\end{array}
\end{equation}
in the sense of distributions.
Since by Proposition \ref{fokker planck bounds weak} we have $m_1,m_2 \in L^2(0,T;H^1(\Omega)) \cap L^\infty(Q)$, we can use $m_1-m_2$ as a test function in \eqref{renormalized hamilton jacobi} to get
\begin{multline} \label{k integration by parts}
-\int_Q [S_k''(u_1)|\nabla u_1|^2 -S_k''(u_2)|\nabla u_2|^2][m_1-m_2]
+ \int_{\Omega} [S_k(g(m_1(T)))-S_k(g(m_2(T)))][m_1(T)-m_2(T)] \\
+ \int_Q [S_k'(u_1)f(m_1)-S_k'(u_2)f(m_2)][m_1-m_2] \\
-\int_Q [S_k'(u_1)H(m_1,\nabla u_1)-S_k'(u_2)H(m_2,\nabla u_2)][m_1-m_2] \\
+ \int_Q [m_1\nabla_p H(m_1,\nabla u_1) -  m_2\nabla_p H(m_2,\nabla u_2)] \cdot \nabla [S_k(u_1)-S_k(u_2)]  = 0.
\end{multline}
As $k \to \infty$, $S_k''(u_i)|\nabla u_i|^2 \to 0$ strongly in $L^1(Q)$, for $i=1,2$.
Since $m_1$ and $m_2$ are bounded,
$$
-\int_Q [S_k''(u_1)|\nabla u_1|^2 -S_k''(u_2)|\nabla u_2|^2][m_1-m_2] \to 0, ~~ k \to \infty.
$$
As for the other terms, from the definition of weak solution we know that, for $i=1,2$, $H(m_i,\nabla u_i) \in L^1(Q)$ which implies $\nabla_p H(m_i,\nabla u_i) \in L^{r/(r-1)}(Q)$ by the assumptions on $H$ and the fact that $m_i$ and $m_i^{-1}$ are bounded.
On the other hand, $\nabla u_i \in L^r(Q)$.
So using the fact that $S_k' \to 1$ and $S_k(u) \to u$ and letting $k \to \infty$ in \eqref{k integration by parts}, we get
\begin{multline} \label{uniqueness integration by parts}
\int_Q [f(m_1)-f(m_2)][m_1-m_2]
+ \int_{\Omega} [g(m_1(T))-g(m_2(T))][m_1(T)-m_2(T)] \\
+ \int_Q [m_1\nabla_p H(m_1,\nabla u_1) -  m_2\nabla_p H(m_2,\nabla u_2)] \cdot \nabla [u_1-u_2] -\int_Q [H(m_1,\nabla u_1)-H(m_2,\nabla u_2)][m_1-m_2]  = 0.
\end{multline}
Then we use Lemma \ref{convexity} and the fact that $f$ and $g$ are strictly increasing to deduce that $m_1 = m_2$, $\nabla u_1 = \nabla u_2$, and $u_1(T) = u_2(T)$, from which it follows that $u_1 = u_2$ as well.
This completes the proof.

\section{Mean field type control}

\label{section mean field control}

Consider now the following system, which can be thought of as a companion to \eqref{mfg}.
\begin{equation}\label{mftc}\left\{
\begin{array}{cll}
\vspace{5pt}
(i) & -\partial_t u - \Delta u + H(t,x,m,\nabla u) + m\partial_m H(t,x,m,\nabla u) = f(t,x,m) &\text{in}~~Q  \\
\vspace{5pt}
(ii) & \partial_t m - \Delta m - \mathrm{div} \left(m\nabla_p H(t,x,m,\nabla u)\right) = 0 &\text{in}~~Q  \\
(iii) & u(T,x) = g(x,m(T,x)),~~ m(0,x) = m_0(x) &\text{in}~~\Omega 
\end{array}\right.
\end{equation}
The difference between \eqref{mftc} and \eqref{mfg} is the addition of the term $m\partial_m H(t,x,m,\nabla u)$, which makes the system into a condition of optimality for a control problem.
See, for instance, \cite{bensoussan13a}.
Weak solutions of \eqref{mftc} can be defined in a way analogous to Definition \ref{definition of solutions} for mean field games.

Here we describe the mean field type optimal control problem solved by solutions to \eqref{mftc}.
Let the Lagrangian $L = L(t,x,m,b)$ be given by
$$
L(t,x,m,b) = \sup_{p} p \cdot b - H(t,x,m,-p)
$$
Let $\s{K}$ be the set of all pairs $(m,w) \in (L^2(0,T;H^1(\Omega))_+ \cap L^\infty(Q)) \times L^{r/(r-1)}(Q)$ such that
$$
\partial_t m - \Delta m + \mathrm{div}(w) = 0, ~~ m(0) = m_0
$$
holds in the sense of distributions.
Define
$$
\tilde{L}(t,x,m,w) =
mL\left(t,x,m,\frac{w}{m}\right), ~~ (t,x) \in Q, ~ m > 0, ~~ w \in \bb{R}^d.
$$
Observe that $w \mapsto \tilde{L}(t,x,m,w)$ is obtained by taking the Legendre transform of $p \mapsto mH(t,x,m,-p)$.

In this section we consider the optimal control problem of minimizing the objective functional
\begin{equation}
\label{objective functional}
J(m,w) = \int_0^T \int_\Omega \tilde{L}(t,x,m(t,x),w(t,x)) + F(t,x,m(t,x)) dx dt + \int_{\Omega} G(x,m(T,x))dx
\end{equation}
over all $(m,w) \in \s{K}$.
Here $F$ and $G$ are given by
$$
F(t,x,m) = \int_{0}^{m} f(t,x,s)ds, ~~~~ G(x,m) = \int_{0}^{m} g(x,s)ds.
$$

The assumptions of Section \ref{assumptions} are still in force.
We also add the following hypotheses:
\begin{enumerate}
\item We will assume in addition to the previous assumptions that for all $(t,x) \in Q$, the map $(m,p) \mapsto H(t,x,m,p)$ is $C^1$ on $(0,\infty) \times \bb{R}^d$.
\item We assume also that for all $(t,x) \in Q$, the map $(m,w) \mapsto \tilde{L}(t,x,m,w)$ is $C^1$ and convex on $(0,\infty) \times \bb{R}^d$ and that \eqref{H positive} holds for $\partial_m H$ in place of $H$.
\item Finally, we add the assumption that either $(m,w) \mapsto \tilde{L}(t,x,m,w)$ is strictly convex or else $w \mapsto \tilde{L}(t,x,m,w)$ and $m \mapsto F(t,x,m)$ are both strictly convex.
\end{enumerate}
Let us remark that this last assumption on {\em strict} convexity is required for uniqueness of solutions.
Existence of minimizers will hold even if the convexity is not strict.
\begin{remark}
An example of a Hamiltonian $H$ satisfying the above assumptions is the canonical example
$$
H(t,x,m,p) = \frac{|p|^r}{rm^\alpha}
$$
whenever $\alpha \geq 0$ is any non-negative number.
In particular, we remark that uniqueness holds for the mean field type control system \ref{mftc} for a much larger class of Hamiltonians than for the mean field game \eqref{mfg}.
\end{remark}

In addition, we make the following deductions concerning $F$ and $G$:
\begin{itemize}
\item Since $m \mapsto f(t,x,m)$ and $m \mapsto g(x,m)$ are continuous and nondecreasing, it follows that $m \mapsto F(t,x,m)$ and $m \mapsto G(x,m)$ are $C^1$ and convex.
Moreover, it follows from \eqref{fL and gL defined} that
\begin{equation}
F_L(t,x) := \sup_{m \in [0,L]} F(t,x,m), ~~~ G_L(x) := \sup_{m \in [0,L]} G(x,m)
\end{equation}
are both integrable for all $L > 0$.
\item In addition, we remark that since $f$ and $g$ are bounded below, we have
\begin{equation}
F(t,x,m) \geq -Cm, ~~ G(x,m) \geq -Cm
\end{equation}
for some constant $C > 0$.
\end{itemize}
Because $\tilde{L}$ is strictly convex in $w$, we note that the equivalence of the first-order optimality criteria
\begin{equation}
\label{first order equivalence}
-p = \nabla_w \tilde{L}(t,x,m,w) 
~~ \Leftrightarrow ~~
\tilde{L}(t,x,m,w) + mH(t,x,m,p) = -p \cdot w
~~ \Leftrightarrow ~~
-w = m\nabla_p H(t,x,m,p).
\end{equation}
In particular, if any of these criteria hold, then by taking partial derivatives in $m$ one obtains $H(t,x,m,p) + m\partial_m H(t,x,m,p) = -\partial_m \tilde{L}(t,x,m,w)$.

We now come to the main result of this section.
\begin{theorem} \label{existence for mftc}
Under the assumptions above, and given $T$ small enough, there exists a unique solution $(u,m)$ of \eqref{mftc}.
The pair $(m,-\nabla_p H(t,x,m,\nabla u))$ is the unique minimizer of $J$ in $\s{K}$.
\end{theorem}

The proof of Theorem \ref{existence for mftc} comes in two Lemmas.
\begin{lemma} \label{minimizer exists}
There exists a unique minimizer in $\s{K}$ of $J(m,w)$.
\end{lemma}

\begin{proof}
Let $(m_n,w_n)$ be a sequence in $\s{K}$ such that $J(m_n,w_n)$ converges to $\inf_{(m,w) \in \s{K}} J$.
Set $b_n = w_n/m_n$.
Note that from \eqref{DH bounded} and the definition of $L$ we can deduce that
\begin{equation} \label{L bounded below}
L(t,x,m,b) \geq \frac{|b|^{r/(r-1)}}{C(m^\lambda + m^{-\lambda})} - C(m^\lambda + m^{-\lambda}).
\end{equation}
Since $(m_n,m_nb_n)$ is a minimizing sequence of $J$, then using the estimates on $F$ and $G$ it follows that
\begin{equation} \label{minimizing sequence}
\int_Q m_nL(t,x,m_n,b_n) \leq C + C\int_Q m_n + C\int_{\Omega} m_n(T) \leq C\|m_0\|_1.
\end{equation}
From \eqref{L bounded below} we deduce that, for some $\eta > 0$,
\begin{equation}
|b_n| \leq C(m_n^{\eta} + m_n^{-\eta})(1 + v_n), ~~~ v_n^{r/(r-1)} := m_nL(t,x,m_n,b_n),
\end{equation}
where $\|v_n\|_{r/(r-1)} \leq C$ by \eqref{minimizing sequence}.

Appealing to Proposition \ref{fokker planck bounds}, it follows that $m_n$ is bounded in $L^2(0,T;H^1(\Omega))$ and that $m_n$ and $m_n^{-1}$ are bounded in $L^\infty$.
Passing to a subsequence, we have that $m_n$ converges weakly in $L^2(0,T;H^1(\Omega))$ and pointwise almost everywhere to a function $m \in L^\infty$ such that $1/m$ is also in $L^\infty$.
Moreover, \eqref{L bounded below} and \eqref{minimizing sequence} now imply that $b_n$ is a bounded sequence in $L^{r/(r-1)}$.
Passing to a subsequence we conclude that $b_n$ converges weakly in $L^{r/(r-1)}(Q)$ to some $b$.

We claim $J(m,mb) \leq \liminf_{n \to \infty} J(m_n,m_nb_n)$.
First, notice that for $C > 0$ large enough,
$$
\int_Q F(t,x,m) + Cm \leq \liminf_{n \to \infty} \int_Q F(t,x,m_n) + Cm_n
$$
by Fatou's Lemma.
By taking a subsequence we have $m_n \to m$ strongly in $L^1$, so we get
\begin{equation} \label{F limit}
\int_Q F(t,x,m) \leq \liminf_{n \to \infty} \int_Q F(t,x,m_n).
\end{equation}
A similar argument shows that
\begin{equation}\label{G limit}
\int_{\Omega} G(x,m)  \leq \liminf_{n \to \infty} \int_{\Omega} G(x,m_n).
\end{equation}
Next, because $\tilde{L}$ is convex in $(m,w)$,
\begin{equation} \label{mL limit}
\int_Q \tilde{L}(t,x,m,mb) \leq \liminf_{n \to \infty} \int_Q \tilde{L}(t,x,m_n,m_nb_n).
\end{equation}
Putting \eqref{F limit}, \eqref{G limit}, and \eqref{mL limit} together we see that $J(m,mb) \leq \liminf_{n\to \infty}J(m_n,w_n)$, which implies that $J(m,mb)$ is the minimum over $\s{K}$ of $J$.

To see that $(m,mb)$ is unique, it suffices to note that by the assumptions given above, $J$ is strictly convex.
\end{proof}

\begin{lemma} \label{minimizer is solution}
If $(m,w) \in \s{K}$ is the minimizer of $J$, then $(u,m)$ is a solution of \eqref{mftc}, where $u \in C([0,T];L^1(\Omega))$ solves the adjoint equation
\begin{equation}
\label{adjoint equation}
-\partial_t u - \Delta u = f(m) + \partial_m \tilde{L}(t,x,m,w), ~~ u(T) = g(m(T)).
\end{equation}
Conversely, if $(u,m)$ is a solution of \eqref{mftc}, then $(m,-m\nabla_p H(t,x,m,\nabla u))$ is in $\s{K}$ and is the minimizer of $J$.
\end{lemma}

\begin{proof}
Let $(m,w) \in \s{K}$ be the minimizer of $J$.
By the proof of Lemma \ref{minimizer exists}, $m$ and $1/m$ are bounded.
It follows that $J$ is G\^ateaux differentiable at $(m,w)$.
To see this, observe that for any $(m_1,w_1) \in \s{K}$ we have that $(\tilde{m},\tilde{w}) = (m_1-m,w_1-w)$ is a pair in $L^2(0,T;H^1(\Omega)) \times L^{r/(r-1)}(Q)$ satisfying
\begin{equation}
\label{fokker planck 0}
\partial_t \tilde{m} - \Delta \tilde{m} + \mathrm{div}(\tilde{w}) = 0, ~~ \tilde{m}(0) = 0.
\end{equation}
We calculate
\begin{multline} \label{gateaux 1}
\lim_{\theta \to 0} \frac{1}{\theta}J(m+\theta \tilde{m},w+\theta \tilde{w}) \\
= \int_0^T \int_{\Omega} \partial_m \tilde{L}(t,x,m,w)\tilde{m} + \nabla_w \tilde{L}(t,x,m,w) \cdot \tilde{w} + f(t,x,m)\tilde{m} + \int_{\Omega} g(x,m(T))\tilde{m}(T).
\end{multline}
The right-hand side is well-defined: using \eqref{H positive}, \eqref{L bounded below}, and \eqref{first order equivalence} together with the fact that $m$ and $1/m$ are bounded we arrive at
$$
|\nabla_w \tilde{L}(t,x,m,w)| \leq C|w|^{1/(r-1)} + C,
$$ 
which implies $\nabla_w \tilde{L}(t,x,m,w) \in L^{r}$ because $w \in L^{r/(r-1)}$;
on the other hand, $\tilde{m} \in L^\infty(Q)$ by Proposition \ref{fokker planck bounds weak}, $\partial_m \tilde{L}(t,x,m,w) \in L^1(Q)$ by \eqref{H positive} applied to $\partial_m H$, and $f(t,x,m)$ and $g(x,m(T))$ are integrable because $m$ is bounded.

Introduce the adjoint variable $u \in C([0,T];L^1(\Omega))$ by taking the solution of \eqref{adjoint equation},
where we note that the right-hand side is in $L^1(Q)$.
Then we have, by \eqref{fokker planck 0},
\begin{multline} \label{gateaux 2}
\lim_{\theta \to 0} \frac{1}{\theta}J(m+\theta \tilde{m},w+\theta \tilde{w}) = \int_0^T \int_{\Omega} (-\partial_t u - \Delta u)\tilde{m} + \nabla_w \tilde{L}(t,x,m,w) \cdot \tilde{w} + \int_{\Omega} g(x,m(T))\tilde{m}(T)
\\ = \int_0^T \int_{\Omega} u(\partial_t \tilde{m} - \Delta \tilde{m}) + \nabla_w \tilde{L}(t,x,m,w) \cdot \tilde{w}
= \int_0^T \int_{\Omega} \left(\nabla u + \nabla_w \tilde{L}(t,x,m,w)\right) \cdot \tilde{w}.
\end{multline}
By optimality of $(m,w)$, the value of this G\^ateaux derivative is zero.
Since $\tilde{w}$ is arbitrary, it follows that
$$
-\nabla u = \nabla_w \tilde{L}(t,x,m,w),
$$
and so by \eqref{first order equivalence} we get
\begin{equation}
\label{H and L}
H(t,x,m,\nabla u) + m\partial_m H(t,x,m,\nabla u) = -\partial_m \tilde{L}(t,x,m,w).
\end{equation}
Therefore by substitution $u$ solves
$$
-\partial_t u - \Delta u + H(t,x,m,\nabla u) + m\partial_m H(t,x,m,\nabla u) = f(t,x,m), ~~ u(T) = g(x,m(T)),
$$
while $m$ solves
$$
\partial_t m - \Delta m - \mathrm{div}(m\nabla_p H(t,x,m,p)) = 0, ~~ m(0) = m_0,
$$
as desired.

Finally, we show that solutions of \eqref{mftc} are equivalent to minimizers of $J$.
Let $(u,m)$ be a weak solution of \eqref{mftc}.
Set $w = -m\nabla_pH(t,x,m,p)$.
Note that $(m,w) \in \s{K}$.
Following the reasoning in the previous step, we obtain that $\nabla u = -\nabla_w \tilde{L}(t,x,m,w)$ and that \eqref{H and L} holds.
Then by the calculations in \eqref{gateaux 1} and \eqref{gateaux 2} we get
\begin{equation}
\lim_{\theta \to 0} \frac{1}{\theta}J(m+\theta \tilde{m},w+\theta{w}) = \int_{0}^T \int_{\Omega} \left(\nabla u + \nabla_w \tilde{L}(t,x,m,w)\right) \cdot \tilde{w} = 0,
\end{equation}
hence the G\^ateaux derivative of $J$ at $(m,w)$ is zero.
By the strict convexity of $J$, it follows that $(m,w)$ is a minimizer of $J$.
\end{proof}

\begin{proof}
[Proof of Theorem \ref{existence for mftc}]
By Lemmas \ref{minimizer exists} and \ref{minimizer is solution}, there exists a solution $(u,m)$ to \eqref{mftc}.
On the other hand, if $(\tilde{u},\tilde{m})$ is another solution to \eqref{mftc}, then  $(\tilde{m},-\nabla_p H(t,x,\tilde{m},\nabla \tilde{u}))$ is the unique minimizer of $J$, so $\tilde{m} = m$.
Since $u$ and $\tilde{u}$ both satisfy the adjoint equation \eqref{adjoint equation}, it follows that $u = \tilde{u}$ as well.
\end{proof}

To conclude, we remark on some possible generalizations of Theorem \ref{existence for mftc}.
The existence of minimizers for the functional \eqref{objective functional} will hold under much broader assumptions than those presented above.
In this article the hypotheses leading to the use of Proposition \ref{fokker planck bounds} ensure the G\^ateaux differentiability of \eqref{objective functional}, allowing a classical optimal control argument yielding the solutions of \eqref{mftc}.
For cases where the a priori bounds in Proposition \ref{fokker planck bounds} are not applicable, one may instead use the strategy adopted in \cite{cardaliaguet2014second} using Fenchel-Rockafellar duality.
Accordingly, one seeks to prove the existence of minimizers both for the optimal control problem $\min_{(m,w) \in \s{K}} J(m,w)$ and its dual, the latter of which is solved in a suitably relaxed setting.
In this way one overcomes the lack of differentiability of $J(m,w)$ due to its singular dependence on $m$.
This method can also be used to construct solutions of \eqref{mftc} under less restrictive conditions on the Hamiltonian.
We can expect this to be the subject of future research.

\bibliographystyle{siam}
\bibliography{C:/mybib/mybib}

\begin{thebibliography}{10}

\bibitem{bensoussan13a}
{\sc A.~Bensoussan, J.~Frehse, and P.~Yam}, {\em Mean field games and mean
  field type control theory}, SpringerBriefs in Mathematics, Springer.

\bibitem{burger2013mean}
{\sc M.~Burger, M.~Di~Francesco, P.~Markowich, and M.-T. Wolfram}, {\em Mean
  field games with nonlinear mobilities in pedestrian dynamics}, arXiv preprint
  arXiv:1304.5201,  (2013).

\bibitem{cardaliaguet2013weak}
{\sc P.~Cardaliaguet}, {\em Weak solutions for first order mean field games
  with local coupling}, arXiv preprint arXiv:1305.7015,  (2013).

\bibitem{cardaliaguet2012geodesics}
{\sc P.~Cardaliaguet, G.~Carlier, and B.~Nazaret}, {\em Geodesics for a class
  of distances in the space of probability measures}, Calculus of Variations
  and Partial Differential Equations,  (2012), pp.~1--26.

\bibitem{cardaliaguet2014second}
{\sc P.~Cardaliaguet, J.~Graber, A.~Porretta, and D.~Tonon}, {\em Second order
  mean field games with degenerate diffusion and local coupling}, arXiv
  preprint arXiv:1407.7024,  (2014).

\bibitem{cardaliaguet2014mean}
{\sc P.~Cardaliaguet and P.~J. Graber}, {\em Mean field games systems of first
  order}, ESAIM: Control, Optimisation, and Calculus of Variations (to appear),
   (2014).

\bibitem{dolbeault2009new}
{\sc J.~Dolbeault, B.~Nazaret, and G.~Savar{\'e}}, {\em A new class of
  transport distances between measures}, Calculus of Variations and Partial
  Differential Equations, 34 (2009), pp.~193--231.

\bibitem{gomes2014existence}
{\sc D.~A. Gomes and H.~Mitake}, {\em Existence for stationary mean field games
  with quadratic hamiltonians with congestion}, arXiv preprint arXiv:1407.8267,
   (2014).

\bibitem{gomes2015short}
{\sc D.~A. Gomes and V.~K. Voskanyan}, {\em Short-time existence of solutions
  for mean field games with congestion}, preprint.

\bibitem{graber2014optimal}
{\sc P.~J. Graber}, {\em Optimal control of first-order hamilton--jacobi
  equations with linearly bounded hamiltonian}, Applied Mathematics \&
  Optimization, 70 (2014), pp.~185--224.

\bibitem{ladyzhenskaia1968linear}
{\sc O.~A. Ladyzhenskai͡a, V.~A. Solonnikov, and N.~N. Ural'ceva}, {\em Linear
  and Quasi-linear Equations of Parabolic Type}, vol.~23, American Mathematical
  Soc., 1968.

\bibitem{lasry07}
{\sc J.-M. Lasry and P.-L. Lions}, {\em Mean field games}, Japanese Journal of
  Mathematics, 2 (2007), pp.~229--260.

\bibitem{porretta1999existence}
{\sc A.~Porretta}, {\em Existence results for nonlinear parabolic equations via
  strong convergence of truncations}, Annali di Matematica Pura ed Applicata,
  177 (1999), pp.~143--172.

\bibitem{porretta2013weak}
\leavevmode\vrule height 2pt depth -1.6pt width 23pt, {\em Weak solutions to
  {Fokker--Planck} equations and mean field games}, Archive for Rational
  Mechanics and Analysis,  (2013), pp.~1--62.

\bibitem{simon86}
{\sc J.~Simon}, {\em Compact sets in the space {$L^p(0,T;B)$}}, Annali di
  Matematica pura ed applicata, 146 (1986), pp.~65--96.

\end{thebibliography}

\end{document}